\documentclass[reqno,12pt]{article}
\usepackage{amsfonts,amssymb,amsmath,amsthm,epsfig,mathrsfs}
\usepackage{xcolor}

\newcommand{\N}{\mathbb{N}}
\newcommand{\R}{\mathbb{R}}

\newcommand{\res}{\mathop{\hbox{\vrule height 7pt width .5pt depth 0pt
\vrule height .5pt width 6pt depth 0pt}}\nolimits}

\newcommand{\Haus}[1]{{\mathscr S}^{#1}} 
\newcommand{\Leb}[1]{{\mathscr L}^{#1}} 

\newcommand{\eps}{\varepsilon}

\newtheorem{theorem}{Theorem}[section]
\newtheorem{lemma}[theorem]{Lemma}
\newtheorem{definition}[theorem]{Definition}
\newtheorem{remark}[theorem]{Remark}
\newtheorem{proposition}[theorem]{Proposition}
\newtheorem{corollary}[theorem]{Corollary}
\newtheorem{ex}[theorem]{Example}

\title{Surface measures and convergence of the
Ornstein-Uhlenbeck semigroup in Wiener spaces}
\author{Luigi Ambrosio\footnote{Scuola Normale Superiore,
p.za dei Cavalieri 7, I-56126 Pisa, Italy, e--mail:
l.ambrosio@sns.it}, Alessio Figalli\footnote{e--mail:
figalli@math.utexas.edu}}

\begin{document}

\maketitle

\begin{abstract}
We study points of density $1/2$ of sets of finite perimeter in
infinite-dimensional Gaussian spaces and prove that, as in the
finite-dimensional theory, the surface measure is concentrated on
this class of points. Here density $1/2$ is formulated in terms of
the pointwise behaviour of the Ornstein-Uhlembeck semigroup.
\end{abstract}

\section{Introduction}

The theory of sets of finite perimeter and $BV$ functions in Wiener
spaces, i.e., Banach spaces endowed with a Gaussian Borel
probability measure $\gamma$, has been initiated by Fukushima and
Hino in \cite{fuk99,fuk2000_1,fuk2000_2}. More recently, some basic
questions of the theory have been investigated in \cite{hin09set}
and in \cite{AMMP,AMP} (see also \cite{ADP} for a slightly different
framework). One motivation for this theory is the development of
Gauss-Green formulas in infinite-dimensional domains; as in the
finite-dimensional theory, it turns out that for nonsmooth domains
the surface measure might be supported in a set much smaller than
the topological boundary (see also the precise analysis made in
\cite{Zambotti}, in a particular class of infinite-dimensional
domains).

The basic question we would like to consider is the research of
infinite-dimensional analogues of the classical fine properties of
$BV$ functions and sets of finite perimeter in finite-dimensional
(Gaussian) spaces.

For this reason we start first with a discussion of the
finite-dimensional theory, referring to \cite{fed} and \cite{afp}
for much more on this subject. Recall that a Borel set
$E\subset\R^m$ is said to be of \textit{finite perimeter} if there exists a
vector valued measure $D\chi_E=(D_1\chi_E,\ldots,D_m\chi_E)$ with
finite total variation in $\R^m$ satisfying the integration by parts
formula:
\begin{equation}\label{byparts}
\int_E\frac{\partial\phi}{\partial
x_i}\,dx=-\int_{\R^m}\phi\,dD_i\chi_E\qquad \forall i=1,\ldots,m,
\,\,\forall\phi\in C^1_c(\R^m).
\end{equation}
De~Giorgi proved in \cite{DeG1} a deep result on the structure of
$D\chi_E$. First of all he identified a set ${\mathcal F}E$, called
by him \textit{reduced boundary}, on which $|D\chi_E|$ is
concentrated, and defined a pointwise inner normal
$\nu_E(x)=(\nu_{E,1}(x),\ldots,\nu_{E,m}(x))$ (see \eqref{dgredbou});
then, through a
suitable blow-up procedure, he proved that ${\mathcal F}E$ is
countably rectifiable (more precisely, it is contained in the union of
countably many graphs of Lipschitz functions defined on hyperplanes of
$\R^m$); finally, he proved the representation formula
$D\chi_E=\nu_E\Haus{m-1}\res{\mathcal F}E$, where $\Haus{m-1}$ is
the $(m-1)$-dimensional spherical Hausdorff measure in $\R^m$. In
light of these results, the integration by parts formula reads
$$
\int_E\frac{\partial\phi}{\partial x_i}\,dx=-\int_{{\mathcal
F}E}\phi\nu_{E,i}\,d\Haus{m-1}\qquad \forall i=1,\ldots,m,
\,\,\forall\phi\in C^1_c(\R^m).
$$
A few years later, Federer proved in \cite{fed1} that the same
representation result of $D\chi_E$ holds for another concept of
boundary, called \textit{essential boundary}:
$$
\partial^*E:=\left\{x\in\R^m:\ \limsup_{r\downarrow
0}\frac{\Leb{m}(B_r(x)\cap E)}{\Leb{m}(B_r(x))}>0,\qquad
\limsup_{r\downarrow 0}\frac{\Leb{m}(B_r(x)\setminus
E)}{\Leb{m}(B_r(x))}>0 \right\},
$$
where $\Leb{m}$ is the $m$--dimensional Lebesgue measure (this
corresponds to points neither of density 0, nor of density 1).
Indeed, a consequence of the De~Giorgi's blow-up procedure is that
${\mathcal F}E\subset \partial^* E$ (because tangent sets to $E$ at
all points in the reduced boundary are halfspaces, whose density at
the origin is $1/2$), and in \cite{fed1} it is shown that
$\Haus{m-1}(\partial^*E\setminus {\mathcal F}E)=0$. Since the set
$E^{1/2}$ of points of density $1/2$
$$
E^{1/2}:=\left\{x\in\R^m:\ \lim_{r\downarrow
0}\frac{\Leb{m}(B_r(x)\cap E)}{\Leb{m}(B_r(x))}=\frac{1}{2}\right\},
$$
is in between the two, one can also use it as a good definition of
boundary.

When looking for the counterpart of De~Giorgi's and Federer's
results in infinite-dimensional spaces, one can consider a suitable
notion of ``distributional derivative'' along Cameron-Martin
directions $D_\gamma\chi_E$ and surface measure $|D_\gamma\chi_E|$.
But, several difficulties arise:
\begin{itemize}
\item[(i)] The classical concept of Lebesgue approximate continuity,
underlying also the definition of essential boundary, seems to fail
or seems to be not reproducible in Gaussian spaces $(X,\gamma)$. For
instance, in \cite{Preiss} it is shown that in general the balls of
$X$ cannot be used, and in any case the norm of $X$ is not natural
from the point of view of the calculus in Wiener spaces, where no
intrinsic metric structure exists and the ``differentiable''
structure is induced by $H$.
\item[(ii)] Suitable notions of codimension-1 Hausdorff measure,
of rectifiability and of essential/reduced boundary have to be
devised.
\end{itemize}
Nevertheless, some relevant progresses have been obtained by
Feyel-De la Pradelle in \cite{feydelpra}, by Hino in \cite{hin09set}
and, on the rectifiability issue, by the first author, Miranda and
Pallara in \cite{AMP}. In \cite{feydelpra} a family of spherical
Hausdorff pre-measures ${\mathscr S}^{\infty-1}_F$ has been
introduced by looking at the factorization $X={\rm
Ker}(\Pi_F)\otimes F$, with $F$ $m$-dimensional subspace of $H$,
considering the measures ${\mathscr S}^{m-1}$ on the $m$-dimensional
fibers of the decomposition. A crucial monotonicity property of
these pre-measures with respect to $F$ allows to define ${\mathscr
S}^{\infty-1}_{FDP}$ (here, $FDP$ stands for Feyel-De la Pradelle)
as $\lim_F{\mathscr S}^{\infty-1}_F$, the limit being taken in the
sense of directed sets. This Hausdorff measure, when restricted to
the boundary of a ``nice'' set (in the sense of Malliavin calculus)
is then shown to be consistent with the surface measure defined in
\cite{aim}. In \cite{hin09set} this approach has been used to build
a Borel set $\partial^*_{{\mathcal F}}E$, called cylindrical
essential boundary, for which the representation formula
\begin{equation}\label{rapphino}
|D_\gamma\chi_E|={\mathscr S}^{\infty-1}_{{\mathcal
F}}\res\partial^*_{{\cal F}}E
\end{equation}
holds. Here ${\cal F}=\{F_n\}_{n\geq 1}$ is an nondecreasing family
of finite-dimensional subspaces of $\tilde{H}$ (see \eqref{defhstar} for
the definition of $\tilde{H}$) whose union is dense in $H$ and ${\mathscr
S}^{\infty-1}_{{\mathcal F}}=\lim_n{\mathscr S}^{\infty-1}_{F_n}$.
Notice that, while the left hand side in the representation formula
is independent of the choice of ${\mathcal F}$, both the cylindrical
essential boundary and ${\mathscr S}^{\infty-1}_{{\mathcal F}}$ a
priori depend on ${\mathcal F}$ (see Remark~\ref{rcomparehino} for a
more detailed discussion). The problem of getting a representation
formula in terms of a coordinate-free measure ${\mathscr
S}^{\infty-1}$ is strongly related to the problem of finding
coordinate-free definitions of reduced/essential boundary.

In this paper, answering in part to questions raised in
\cite{hin09set} and in \cite{AMP}, we propose an
infinite-dimensional counterpart of $E^{1/2}$ and use it to provide
a coordinate-free version of \eqref{rapphino}.

In view of the quite general convergence results illustrated in
\cite{Stein} it is natural, in this context, to think of the
Ornstein-Uhlenbeck semigroup $T_t\chi_E$ starting from $\chi_E$, for
small $t$, as an analog of the mean value of $\chi_E$ on small
``balls''. Also, it is already known starting from \cite{DeG0} (see
also \cite{fuk2000_1,fuk2000_2,AMMP,ledu1}) that surfaces measures
are intimately connected to the behavior of $T_t\chi_E$ for small
$t$. Our first main result provides strong convergence of
$T_t\chi_E$ as $t\downarrow 0$, if we take the surface measure as
reference measure:
\begin{theorem}\label{tessbou1}
Let $E$ be a Borel set of finite perimeter in $(X,\gamma)$. Then
$$
\lim_{t\downarrow
0}\int_X|T_t\chi_E-\frac{1}{2}|^2\,d|D_\gamma\chi_E|=0.
$$
\end{theorem}
Since $|D_\gamma\chi_E|$ is orthogonal w.r.t. $\gamma$, it is
crucial for the validity of the result that $T_t\chi_E$ is not
understood in a functional way (i.e., as an element of
$L^\infty(X,\gamma)$), but really in a pointwise way through
Mehler's formula \eqref{mehler}. In this respect, the choice of a
Borel representative is important, see also
Proposition~\ref{pbogaregu} and \eqref{product}.

The proof of Theorem~\ref{tessbou1} is based on two results: first, by a soft
argument based on the product rule for weak derivatives, we show
the weak$^*$ convergence of $T_t\chi_E$ to $1/2$ in
$L^\infty(X,|D_\gamma\chi_E|)$. Then, by  a quite delicate finite-dimensional
approximation and factorization of the OU semigroup,
we show the apriori estimate
$$
\limsup_{t\downarrow
0}\int_X|T_t\chi_E|^2\,d|D_\gamma\chi_E|^2\leq\frac
{1}{4}|D_\gamma\chi_E|(X).
$$
Notice that in
finite dimensions Theorem~\ref{tessbou1} is easy to show, using the
fact that sets of finite perimeter are, for $|D_\gamma\chi_E|$-a.e.
$x$, close to halfspaces on small balls centered at $x$ (see the
proof of Proposition~\ref{findim1} and also Remark~\ref{ralessio}).

Thanks to Theorem~\ref{tessbou1}, we can choose an infinitesimal
sequence $(t_i)\downarrow 0$ such that
\begin{equation}\label{hino9}
 \sum_i\int_X|T_{t_i}\chi_E-\frac12|\,d|D_\gamma\chi_E|<\infty,
\end{equation}
This choice of $(t_i)$ ensures in particular the convergence of $T_{t_i}\chi_E$ to
$1/2$ $|D_\gamma\chi_E|$-a.e. in $X$, and motivates the next
definition:
\begin{definition}[Points of
density $1/2$]\label{disoessbou} Let $(t_i)\downarrow 0$ be such
that $\sum_i\sqrt{t_i}<\infty$ and \eqref{hino9} holds. We denote by
$E^{1/2}$ the set
\begin{equation}\label{hino10}
E^{1/2}:=\left\{x\in X:\
\lim_{i\to\infty}T_{t_i}\chi_E(x)=\frac12\right\}.
\end{equation}
\end{definition}
Notice that $|D_\gamma\chi_E|$ is concentrated on
$E^{1/2}$. With this definition, and defining $\Haus{\infty-1}$ as
the supremum of $\Haus{\infty-1}_F$ among all finite-dimensional
subspaces of $\tilde{H}$, we can prove our second main result:
\begin{theorem}\label{tessbou2}
Let $(t_i)\downarrow 0$ be such that $\sum_i\sqrt{t_i}<\infty$ and
\eqref{hino9} holds. Then the set $E^{1/2}$ defined in \eqref{hino10} has
finite $\Haus{\infty-1}$-measure and
\begin{equation}\label{hino31}
|D_\gamma\chi_E|=\Haus{\infty-1}\res E^{1/2}.
\end{equation}
\end{theorem}

As we said, an advantage of \eqref{hino31} is its coordinate-free
character, see also Remark~\ref{rcomparehino} for a more detailed
comparison with Hino's cylindrical definition of essential boundary.
A drawback is its dependence on $(t_i)$; however, this dependence
enters only in the definition of $E^{1/2}$, and not in the one of $\Haus{\infty-1}$. Moreover,
it readily follows
from Theorem~\ref{tessbou2} that $E^{1/2}$ is uniquely determined up
to $\Haus{\infty-1}$-negligible sets (i.e., different sequences
produce equivalent sets). We consider merely as a (quite) technical
issue the replacement of $\Haus{\infty-1}$ with the larger measure
$\Haus{\infty-1}_{FDP}$ (defined considering \emph{all}
finite-dimensional subspaces of $H$) in \eqref{hino31}, for the
reasons explained in Remark~\ref{rtoomuch}.

As an example of application of the structure result for
$|D_\gamma\chi_E|$ provided by \eqref{hino31}, we can provide a
precise formula for the distributional derivative of the union of
two disjoints sets of finite perimeter. Given a set $E$ of finite
perimeter, write $D_\gamma \chi_E=\nu_E |D_\gamma \chi_E|$, with
$\nu_E:X\to H$ a Borel vectorfield satisfying $|\nu_E|_H=1$
$|D_\gamma \chi_E|$-a.e. in $X$. With this notation we have:
\begin{corollary}
\label{cor:product}
Let $E$ and $F$ be sets of finite perimeter with $\gamma(E\cap F)=0$. Then $E\cup F$ has finite perimeter,
\begin{equation}\label{tesi lemma unione}
\nu_{E\cup F}\Haus{\infty-1}\res (E\cup F)^{1/2}=
\nu_E\Haus{\infty-1}\res(E^{1/2}\setminus F^{1/2})+\nu_F\Haus{\infty-1}\res( F^{1/2}\setminus E^{1/2}),
\end{equation}
and $\nu_E(x)=-\nu_F(x)$ at $\Haus{\infty-1}$-a.e. $x\in E^{1/2}\cap F^{1/2}$.
\end{corollary}
An important feature in the above result is that, since $(E\cup
F)^{1/2}$, $E^{1/2}$, and $F^{1/2}$ are uniquely determined up to
$\Haus{\infty-1}$-negligible sets, one does not have to specify
which sequences $(t_i)$ one uses to define the sets (and the
sequences could all be different). On the other hand, if one would
try to deduce the analogous result stated in terms of cylindrical
boundaries, it seems to us that one would be obliged to choose the
same family ${\cal F}=\{F_n\}_{n\geq 1}$ for all the three sets (see
Remark~\ref{rcomparehino}).

Let us conclude this introduction pointing out that our results can
be considered as the analogous of Federer's result to an infinite
dimensional setting. In \cite[Section 7]{AMP}, the authors gave a
list of some open problems related to the rectifiability result, and
gave potential alternative definitions of essential and reduced
boundary. As we will show in the appendix, the approach used in
Proposition~\ref{pweakstar} to prove the weak$^*$ convergence of
$T_t\chi_E$ to $1/2$ in $L^\infty(X,|D_\gamma\chi_E|)$ is flexible
enough to give a ``weak form'' of the fact that $|D_\gamma\chi_E|$
is concentrated also on a kind of reduced boundary. Apart from this,
many other natural questions remain open. In particular, the main
open problem is still to find some analogous of De Giorgi's blow-up
theorem (i.e., understanding in which sense, for
$|D_\gamma\chi_E|$-a.e. $x \in X$, the blow-up of $E$ around $x$
converges to an half-space, see the proof of
Proposition~\ref{findim1}).

\section{Notation and preliminary results}

We assume that $(X,\|\cdot\|)$ is a separable Banach space and
$\gamma$ is a Gaussian probability measure on the Borel
$\sigma$-algebra of $X$. We shall always assume that $\gamma$ is
nondegenerate (i.e., all closed proper subspaces of $X$ are
$\gamma$-negligible) and centered (i.e., $\int_X x\,d\gamma=0$). We
denote by $H$ the Cameron-Martin subspace of $X$, that is
$$
H:=\left\{\int_X f(x)x\,d\gamma(x):f\in L^2(X,\gamma)\right\},
$$
and, for $h\in H$,
we denote by $\hat{h}$ the corresponding element in $L^2(X,\gamma)$;
it can be characterized as the Fomin derivative of $\gamma$ along
$h$, namely
\begin{equation}\label{Fomin}
\int_X\partial_h\phi\,d\gamma=-\int_X\hat{h}\phi\,d\gamma
\end{equation}
for all $\phi\in C^1_b(X)$. Here and in the sequel $C^1_b(X)$
denotes the space of continuously differentiable cylindrical
functions in $X$, bounded and with a bounded gradient. The space $H$
can be endowed with an Hilbertian norm $|\cdot |_H$ that makes the
map $h\mapsto\hat{h}$ an isometry; furthermore, the injection of
$(H,|\cdot |_H)$ into $(X,\|\cdot\|)$ is compact.

We shall denote by $\tilde{H}\subset H$ the subset of vectors of the form
\begin{equation}\label{defhstar}
\int_X \langle x^*,x\rangle x\,d\gamma(x),\qquad x^*\in X^*.
\end{equation}
This is a dense (even w.r.t. to the Hilbertian norm) subspace of
$H$. Furthermore, for $h\in H^*$ the function $\hat{h}(x)$ is
precisely $\langle x^*,x\rangle$ (and so, it is continuous).

Given a $m$-dimensional subspace $F\subset \tilde{H}$ we shall frequently
consider an orthonormal basis $\{h_1,\ldots,h_m\}$ of $F$ and the
factorization $X=F\oplus Y$, where $Y$ is the kernel of the
continuous linear map
\begin{equation}\label{ammiss1}
x\in X\mapsto \Pi_F(x):=\sum_{i=1}^m\hat{h}_i(x)h_i\in F.
\end{equation}
The decomposition $x=\Pi_F(x)+(x-\Pi_F(x))$ is well defined, thanks
to the fact that $\Pi_F\circ \Pi_F=\Pi_F$ and so $x-\Pi_F(x)\in Y$;
in turn this follows by
$\hat{h}_i(h_j)=\langle\hat{h}_i,\hat{h}_j\rangle_{L^2}=\delta_{ij}$.

Thanks to the fact that $|h_i|_H=1$, this induces a factorization
$\gamma=\gamma_F\otimes\gamma_Y$, with $\gamma_F$ the standard Gaussian
in $F$ (endowed with the metric inherited from $H$) and $\gamma_Y$
Gaussian in $(Y,\|\cdot\|)$. Furthermore, the orthogonal complement
$F^\perp$ of $F$ in $H$ is the Cameron-Martin space of
$(Y,\gamma_Y)$.

\subsection{$BV$ functions and Sobolev spaces}

Here we present the definitions of Sobolev and $BV$ spaces. Since we
will consider bounded functions only, we shall restrict to this
class for ease of exposition.

Let $u:X\to\R$ be a bounded Borel function. Motivated by
\eqref{Fomin}, we say that $u\in W^{1,1}(X,\gamma)$ if there exists
a (unique) $H$-valued function, denoted by $\nabla u$, with $|\nabla
u|_H\in L^1(X,\gamma)$ and
$$
\int_X u\partial_h\phi\,d\gamma=-\int_X \phi\langle\nabla
u,h\rangle_H\,d\gamma+\int_X u\phi\hat{h}\,d\gamma
$$
for all $\phi\in C^1_b(X)$ and $h\in H$.

Analogously, following \cite{fuk2000_1,fuk2000_2}, we say that $u\in
BV(X,\gamma)$ if there exists a (unique) $H$-valued Borel measure
$D_\gamma u$ with finite total variation in $X$ satisfying
$$
\int_X u\partial_h\phi\,d\gamma=-\int_X \phi\,d\langle D_\gamma
u,h\rangle_H+\int_X u\phi\hat{h}\,d\gamma
$$
for all $\phi\in C^1_b(X)$ and $h\in H$.

In the sequel, shall mostly consider
the case when $u=\chi_E:X\to\{0,1\}$ is the characteristic function
of a set $E$, although some statements are more natural in the
general $BV$ context. Notice the inclusion $W^{1,1}(X,\gamma)\subset
BV(X,\gamma)$, given by the identity $D_\gamma u=\nabla u\gamma$.

\subsection{The OU semigroup and Mehler's formula}

In this paper, the Ornstein-Uhlenbeck semigroup $T_tf$ will always
be understood as defined by the \emph{pointwise} formula
\begin{equation}\label{mehler}
T_tf(x):=\int_X f(e^{-t}x+\sqrt{1-e^{-2t}}y)\,d\gamma(y)
\end{equation}
which makes sense whenever $f$ is bounded and Borel. This convention
will be important when integrating $T_t f$ against potentially
singular measures, see for instance \eqref{product}.

We shall also use the dual OU semigroup $T_t^*$, mapping signed
measures into signed measures, defined by the formula
\begin{equation}
\langle T_t^*\mu,\phi\rangle:=\int_X T_t\phi\,d\mu\qquad\text{$\phi$
bounded Borel.}
\end{equation}

In the next proposition we collect a few properties of the OU
semigroup needed in the sequel (see for instance \cite{boga} for the
Sobolev case and \cite{AMP} for the $BV$ case).

\begin{proposition}\label{pammiss1} Let $u:X\to\R$ be bounded and Borel and $t>0$.
Then $T_tu\in W^{1,1}(X,\gamma)$ and:
\begin{itemize}
\item[(a)] if $u\in W^{1,1}(X,\gamma)$ then, componentwise, it holds $\nabla T_tu=e^{-t}T_t\nabla u$;
\item[(b)] if $u\in BV(X,\gamma)$ then, componentwise, it holds $\nabla T_tu\gamma=e^{-t}T_t^*(D_\gamma
u)$.
\end{itemize}
\end{proposition}

The next result is basically contained in
\cite[Proposition~5.4.8]{boga}, we state and prove it because we
want to emphasize that the regular version of the restriction of
$T_tf$ to $y+F$, $y\in Y$, provided by the Proposition, is for
$\gamma_Y$-a.e. $y$ precisely the one pointwise defined in Mehler's
formula.

\begin{proposition} \label{pbogaregu} Let $u$ be a bounded Borel function and $t>0$.
With the above notation, for $\gamma_Y$-a.e. $y\in Y$ the map
$z\mapsto T_t u(z,y)$ is smooth in $F$.
\end{proposition}
\begin{proof} Let us prove, for the sake of simplicity,
Lipschitz continuity (in fact, the only property we shall need) for
$\gamma_Y$-a.e. $y$, with a bound on the Lipschitz constant
depending only on $t$ and on the supremum of $|u|$. We use the
formula
$$
\partial_h T_tu(x)=\frac{e^{-t}}{\sqrt{1-e^{-2t}}}\int_X
u(e^{-t}x+\sqrt{1-e^{-2t}}y)\hat{h}(y)\,d\gamma(y)\qquad h\in H
$$
for the weak derivative and notice that, if $u$ is cylindrical, this
provides also the classical derivative. On the other hand, the
formula provides also the uniform bound $\sup |\partial_hT_tu|\leq
c(t)|h|_H\sup|u|$. The uniform bound and Fubini's theorem ensure
that the class of functions for which the stated property is true
contains all cylindrical functions and it stable under pointwise
equibounded limits. By the monotone class theorem, the stated
property holds for all bounded Borel functions.\end{proof}

The next lemma provides a rate of convergence of $T_t u$ to $u$ when
$u$ belongs to $BV(X,\gamma)$; the proof follows the lines of the
proof of Poincar\'e inequalities, see \cite[Theorem~5.5.11]{boga}.

\begin{lemma}\label{lpoincare} Let $u\in BV(X,\gamma)$. Then
$$
\int_X|T_tu-u|\,d\gamma\leq c_t|D_\gamma u|(X)
$$
with
$c_t:=\sqrt{\frac{2}{\pi}}\int_0^t\frac{e^{-s}}{\sqrt{1-e^{-2s}}}\,ds$,
$c_t\sim 2\sqrt{t/\pi}$ as $t\downarrow 0$.
\end{lemma}
\begin{proof} It obviously suffices to bound with
$c_t|D_\gamma u|(X)$ the expression
\begin{equation}\label{hino11}
\int_X\int_X
|u(x)-u(e^{-t}x+\sqrt{1-e^{-2t}}y)|\,d\gamma(x)d\gamma(y).
\end{equation}
Standard cylindrical approximation arguments reduce the proof to the
case when $u$ is smooth, $X$ is finite-dimensional and $\gamma$ is
the standard Gaussian. Since
\begin{eqnarray*}
u(e^{-t}x+\sqrt{1-e^{-2t}}y)-u(x)&=&\int_0^1\frac{d}{d\tau}
u(e^{-t\tau}x+\sqrt{1-e^{-2t\tau}}y)\,d\tau\\&=& t\int_0^1 \nabla
(e^{-t\tau}x+\sqrt{1-e^{-2t\tau}}y)\cdot \biggl(-e^{-t\tau} x+
\frac{e^{-2t\tau}y}{\sqrt{1-e^{-2t\tau}}}\biggr)\,d\tau
\end{eqnarray*}
we can estimate the expression in \eqref{hino11} with
$$
t\int_0^1\frac{e^{-t\tau}}{\sqrt{1-e^{-2t\tau}}}\int_X\int_X|\nabla
u(e^{-t\tau}x+\sqrt{1-e^{-2t\tau}}y)\cdot (-\sqrt{1-e^{-2t\tau}} x+
e^{-t\tau}y)|\,d\gamma(x)d\gamma(y)d\tau.
$$
Now, for $\tau$ fixed we can perform the ``Gaussian rotation''
$$
(x,y)\mapsto\bigl(e^{-t\tau}x+\sqrt{1-e^{-2t\tau}}y,-\sqrt{1-e^{-2t\tau}}
x+ e^{-t\tau}y\bigr)
$$
to get
$$
t\int_0^1\frac{e^{-t\tau}}{\sqrt{1-e^{-2t\tau}}}\int_X\int_X|\nabla
u(v)\cdot w|\,d\gamma(w)d\gamma(v)d\tau.
$$
Eventually we use the fact that $\int_X|\xi\cdot
w|\,d\gamma(w)=\sqrt{2/\pi}|\xi|$ to get
$$
t\sqrt{\frac{2}{\pi}}\int_0^1\frac{e^{-t\tau}}{\sqrt{1-e^{-2t\tau}}}\,d\tau
\int_X|\nabla u|(v)\,d\gamma(v).
$$
A change of variables leads to the desired expression of
$c_t$.\end{proof}

Notice that the proof of the lemma provides the slightly stronger
information
\begin{equation}\label{poincarestr}
\int_X\int_X
|u(x)-u(e^{-t}x+\sqrt{1-e^{-2t}}y)|\,d\gamma(x)d\gamma(y)\leq c_t
|D_\gamma u|(X).
\end{equation}
This more precise formulation will be crucial in the proof of
Proposition~\ref{papriori14}.

\subsection{Product rule}

In the proof of Proposition~\ref{pweakstar} we shall use the product
rule
$$
D_\gamma (\chi_E v)=\chi_E\nabla v\gamma +v D_\gamma\chi_E
$$
for $v\in W^{1,1}(X,\gamma)$ and $E$ with finite perimeter. In
general the proof of this property is delicate, even in
finite-dimensional spaces, since a precise representative of $v$
should be used to make sense of the product $vD_\gamma\chi_E$.
However, in the special case when $v=T_tf$ with $t>0$ and $f$
bounded Borel, the product rule, namely
\begin{equation}\label{product}
D_\gamma (\chi_E T_tf)=\chi_E\nabla T_tf\gamma +T_tf D_\gamma\chi_E.
\end{equation}
holds provided we understand $T_tf$ as pointwise defined in Mehler's
formula. The argument goes by pointwise approximation by better
maps, very much as in Proposition~\ref{pbogaregu}, and we shall not
repeat it.

\subsection{Factorization of $T_t$ and $D_\gamma u$}

Let us consider the decomposition $X=F\oplus Y$, with $F\subset\tilde{H}$
finite-dimensional. Denoting by $T_t^F$ and $T_t^Y$ the OU
semigroups in $F$ and $Y$ respectively, it is easy to check (for
instance first on products of cylindrical functions on $F$ and $Y$,
and then by linearity and density) that also the action of $T_t$ can
be ``factorized'' in the coordinates $x=(z,y)\in F\times Y$ as
follows:
\begin{equation}\label{factorization}
T_tf(z,y)=T_t^Y\bigl(w\mapsto T_t^F f(\cdot,w)(z)\bigr)(y)
\end{equation}
for any bounded Borel function $f$.

Let us discuss, now, the factorization properties of $D_\gamma u$.
Let us write $D_\gamma u=\nu_u|D_\gamma u|$ with $\nu_u:X\to H$ Borel
vectorfield with $|\nu_u|_H=1$ $|D_\gamma u|$-a.e. Moreover, given a Borel set $B$, define
$$
B_y:=\left\{z\in F:\ (z,y)\in B\right\},\qquad B_z:=\left\{y\in Y:\
(z,y)\in B\right\}.
$$
The identity
\begin{equation}\label{hino440}
\int_B|\pi_F(\nu_u)|\,d|D_\gamma
u|=\int_Y|D_{\gamma_F}u(\cdot,y)|(B_y)\,d\gamma_Y(y) \qquad\text{$B$
Borel}
\end{equation}
is proved in \cite[Theorem~44.2]{AMP} (see also \cite{AMMP,hin09set}
for analogous results), where $\pi_F:H\to F$ is the orthogonal
projection. Along the similar lines, one can also show the identity
\begin{equation}\label{hino441}
\int_B|\pi_{F^\perp}(\nu_u)|\,d|D_\gamma
u|=\int_F|D_{\gamma_Y}u(z,\cdot)|(B_z)\,d\gamma_F(z) \qquad\text{$B$
Borel}
\end{equation}
with $\pi_F+\pi_{F^\perp}={\rm Id}$.
In the particular case $u=\chi_E$, with the notation
\begin{equation}\label{hino16}
E_y:=\left\{z\in F:\ (z,y)\in E\right\},\qquad E_z:=\left\{y\in Y:\
(z,y)\in E\right\}
\end{equation}
the identities \eqref{hino440} and \eqref{hino441} read respectively
as
\begin{equation}\label{hino40}
\int_B|\pi_F(\nu_E)|\,d|D_\gamma
\chi_E|=\int_Y|D_{\gamma_F}\chi_{E_y}|(B_y)\,d\gamma_Y(y)
\qquad\text{$B$ Borel,}
\end{equation}
\begin{equation}\label{hino41}
\int_B|\pi_{F^\perp}(\nu_E)|\,d|D_\gamma
\chi_E|=\int_F|D_{\gamma_Y}\chi_{E_z}|(B_z)\,d\gamma_F(z)
\qquad\text{$B$ Borel}
\end{equation}
with $D_\gamma\chi_E=\nu_E|D_\gamma\chi_E|$.

\begin{remark}\label{rtoomuch}{\rm Having in mind \eqref{hino40} and \eqref{hino41},
it is tempting to think that the formula holds for any orthogonal
decomposition of $H$ (so, not only when $F\subset\tilde{H}$), or even when none of the parts if
finite-dimensional. In order to avoid merely technical complications
we shall not treat this issue here because, in this more general
situation, the ``projection maps'' $x\mapsto y$ and $x\mapsto z$ are
no longer continuous. The problem can be solved removing sets of
small capacity, see for instance \cite{feydelpra} for a more
detailed discussion.}
\end{remark}

As a corollary of the above formulas, we can prove the following
important semicontinuity result for open sets:
\begin{proposition}
\label{prop:sci} For any open set $A \subset X$ the map
$$
u \mapsto |D_\gamma u|(A)
$$
is lower semicontinuous in $BV(X;\gamma)$ with respect to the
$L^1(X,\gamma)$ convergence.
\end{proposition}

\begin{proof}
Let $u_k \to u$ in $L^1(X,\gamma)$.
It suffices to prove the result under the additional assumption that
\begin{equation}
\label{eq:fast L1}
\sum_k \int_X |u_k - u|\,d\gamma<\infty.
\end{equation}
Let $F\subset \tilde H$ be a finite dimensional subspace,
let $X=F\times Y$ be the associated factorization, and use coordinates $x=(z,y) \in F\times Y$ as before.

Thanks to \eqref{eq:fast L1} and Fubini's theorem, $u_k(\cdot,y)\to u(\cdot,y)$ in $L^1(F,\gamma_{F})$
for $\gamma_{Y}$-a.e. $y \in Y$.
Hence, by the lower semicontinuity of the total variation in finite dimensional spaces
(see for instance \cite[Remark 3.5]{afp} for a proof when $\gamma_F$ is replaced by the Lebesgue
measure) we obtain
$$
|D_{\gamma_{F}} u(\cdot,y)|(A_{y}) \leq \liminf_{k\to\infty}
|D_{\gamma_{F}} u_k(\cdot,y)|(A_{y}) \qquad \text{for $\gamma_{Y}$-a.e. $y \in Y$,}
$$
where $A_{y}:=\left\{z\in F:\ (z,y)\in A\right\}$.
Integrating with respect to $\gamma_{Y}$ and using Fatou's lemma we get
$$
\int_{Y}|D_{\gamma_{F}} u(\cdot,y)|(A_{y})\,d\gamma_{Y_n} \leq
\liminf_{k\to\infty} \int_{Y} |D_{\gamma_{F}} u_k(\cdot,y)|(A_{y})\,d\gamma_{Y},
$$
which together with \eqref{hino440} gives
$$
\int_A|\pi_{F}(\nu_u)|\,d|D_\gamma u| \leq \liminf_{k\to \infty}
\int_A|\pi_{F}(\nu_u)|\,d|D_\gamma u_k| \leq \liminf_{k\to \infty}
|D_\gamma u_k|(A)
$$
(recall that $|\nu_u|_H=1$). Since $|\pi_F(\nu_u)| \uparrow 1$ as
$F$ increases to a dense subspace of $H$, we conclude by the
monotone convergence theorem.
\end{proof}

\subsection{Finite-codimension Hausdorff measures}

We start by introducing, following \cite{feydelpra}, pre-Hausdorff
measures which, roughly speaking, play the same role of the
pre-Hausdorff measures $\Haus{n}_\delta$ in the finite-dimensional
theory.

Let $F\subset\tilde{H}$ be finite-dimensional, $m\geq k\geq 0$ and, with
the notation of the previous section, define
\begin{equation}\label{feydel}
\Haus{\infty-k}_F(B):=\int_Y\int_{B_y}
G_m\,d\Haus{m-k}\,d\gamma_Y(y) \qquad\text{$B$ Borel}
\end{equation}
where $m={\rm dim}(F)$ and $G_m$ is the standard Gaussian density in
$F$ (so that $\Haus{\infty-0}_F=\gamma$). It is proved in
\cite{feydelpra} that $y\mapsto\int_{B_y} G_m\,d\Haus{m-k}$ is
$\gamma_Y$-measurable whenever $B$ is Suslin (so, in particular,
when $B$ is Borel), therefore the integral makes sense. The first
key monotonicity property noticed in \cite{feydelpra}, based on
\cite[2.10.27]{fed1}, is
$$
\Haus{\infty-k}_F(B)\leq\Haus{\infty-k}_G(B)\qquad\text{whenever
$F\subset G\subset\tilde{H}$}
$$
provided $\Haus{m-k}$ in \eqref{feydel} is understood as the
\emph{spherical} Hausdorff measure of dimension $m-k$ in $F$. This
naturally leads to the definition
\begin{equation}\label{hino50}
\Haus{\infty-k}(B):=\sup_F\Haus{\infty-k}_F(B)\qquad\text{$B$ Borel,}
\end{equation}
where the supremum runs among all finite-dimensional subspaces $F$
of $\tilde H$. Notice, however, that strictly speaking the measure
defined in \eqref{hino50} does not coincide with the one in
\cite{feydelpra}, since all finite-dimensional subspaces of $H$ are
considered therein. We make the restriction to finite-dimensional
subspaces of $\tilde{H}$ for the reasons explained in
Remark~\ref{rtoomuch}. However, still $\Haus{\infty-k}$ is defined
in a coordinate-free fashion.

These measures have been related for the first time to the perimeter
measure $D_\gamma\chi_E$ in \cite{hin09set}. Hino defined the
$F$-essential boundaries (obtained collecting the essential
boundaries of the finite-dimensional sections $E_y\subset F\times \{y \}$)
\begin{equation}\label{hino17}
\partial_F^*E:=\left\{(z,y):\ z\in\partial^* E_y\right\}
\end{equation}
and noticed another key monotonicity property (see also
\cite[Theorem~5.2]{AMP})
\begin{equation}\label{hino19}
\Haus{\infty-1}_F(\partial^*_FE\setminus\partial^*_GE)=0
\qquad\text{whenever $F\subset G\subset H^*$.}
\end{equation}
Then, choosing a sequence ${\cal F}=\{F_1,F_2,\ldots\}$ of
finite-dimensional subspaces of $H^*$ whose union is dense he
defined
\begin{equation}\label{hino70}
\Haus{\infty-1}_{{\cal F}}:=\sup_n\Haus{\infty-1}_{F_n},\qquad
\partial_{{\cal F}}^*E:=\liminf_{n\to\infty}\partial^*_{F_n}E
\end{equation}
and proved that
\begin{equation}\label{hino60}
|D_\gamma\chi_E|=\Haus{\infty-1}_{{\cal F}}\res\partial_{{\cal
F}}^*E.
\end{equation}

\begin{remark}\label{rcomparehino}
{\rm If we compare \eqref{hino60} with \eqref{hino31}, we see that
both the measure and the set are defined in \eqref{hino31} in a
coordinate-free fashion, using on one hand all finite-dimensional
subspaces of $\tilde H$, on the other hand the OU semigroup. In this
respect, it seems to us particularly difficult to compare null sets
w.r.t. $\Haus{\infty-1}_{{\cal F}}$ and $\Haus{\infty-1}_{{\cal
F}'}$ when ${\cal F}\neq{\cal F}'$; so, even though the left hand
side in \eqref{hino60} is coordinate-free, it seems difficult to
extract from this information a ``universal'' set. On the other
hand, combining \eqref{hino31} and \eqref{hino60} we obtain that
$E^{1/2}$ is equivalent to $\partial^*_{{\cal F}}E$, up to
$\Haus{\infty-1}_{{\cal F}}$-null sets (observe that, on the other
hand, it is not even clear that $\partial_{{\cal F}}^*E$ has
$\Haus{\infty-1}$ finite measure). So, in some sense, $E^{1/2}$ is
``minimal'' against the ``maximal'' measure
$\Haus{\infty-1}$.}\end{remark}

\section{Finite-dimensional facts}

Throughout this section we assume that $(X,\gamma)$ is a
finite-dimensional Gaussian space, with the associated OU semigroup
$T_t$. We assume that the norm of $X$ is equal to the Cameron-Martin
norm, so that we can occasionally identify $X$ with $\R^m$, $m={\rm
dim\,}X$, and identify $\gamma$ with the product $G_m\Leb{m}$ of $m$
standard Gaussians. Give a Borel set $E$, we shall denote by $E^1$
(resp. $E^0$) the set of density points of $E$ (resp. rarefaction
points) with respect to the Lebesgue measure (it would be the same
to consider $\gamma$, since this measure is locally comparable to
$\Leb{m}$).

In this finite dimensional setting, the first result is that the statement of Theorem~\ref{tessbou1} can
be improved, getting pointwise convergence up to
$|D_\gamma\chi_E|$-negligible sets:

\begin{proposition}\label{findim1}
Let $E\subset X$ be with finite $\gamma$-perimeter. Then, as
$t\downarrow 0$, $T_t\chi_E\to 1/2$ $|D_\gamma\chi_E|$-a.e. in $X$.
\end{proposition}
\begin{proof} In this proof we identify $X$ with $\R^m$.
Since $|D_\gamma\chi_E|=G_m|D\chi_E|$, we know that $E$ has locally
finite Euclidean perimeter. Hence, the finite-dimensional theory
ensures that $|D\chi_E|$-almost every point $x$ the rescaled and
translated sets $(E-x)/r$ locally converge in measure as
$r\downarrow 0$ to an halfspace passing through the origin (see for
instance \cite[Theorem~3.59(a)]{afp}). We obtain that for
$|D_\gamma\chi_E|$-almost every point $x$ the sets
$$
E_{t,x}:=\frac{E-e^{-t}x}{\sqrt{1-e^{-2t}}}
$$
locally converge in measure as $t\downarrow 0$ to an halfspace (here
we use the fact that translating by $e^{-t}x$ instead of $x$ is
asymptotically the same, since $1-e^{-t}=o(\sqrt{1-e^{-2t}})$ as
$t\downarrow 0$). Hence, it suffices to show that $T_t\chi_E(x)\to
1/2$ at all points $x$ where this convergence holds. We compute:
\begin{eqnarray*}
T_t\chi_E(x)&=&(2\pi)^{-m/2}\int_{\R^m}\chi_E(e^{-t}x+\sqrt{1-e^{-2t}}y)e^{-|y|^2/2}\,dy\\&=&
(2\pi)^{-m/2}\int_{E_{t,x}}e^{-|y|^2/2}\,dy.
\end{eqnarray*}
Taking the limit as $t\downarrow 0$ yields $(2\pi)^{-m/2}\int_H
e^{-|w|^2/2}\,dw$ for some subspace $H$ with $0\in\partial H$. By
rotation invariance the value of the limit equals $1/2$.\end{proof}

In the next proposition we carefully estimate the blow-up rate of
the density of $T_t^*\mu$ as $t\downarrow 0$ when $\mu$ is a
codimension one Hausdorff measure on a ``nice'' hypersurface.

\begin{proposition}\label{findim2}
Let $K\subset\R^m$ be a Borel set contained in the union of finitely
many $C^1$ compact hypersurfaces. Then, for all $\eps>0$, there exist
$K_\eps \subset K$ and $t_\eps>0$ such that $\Haus{m-1}(K\setminus
K_\eps)<\eps$ and
$$
\sqrt{t}{T_t^*\bigl(G_m\Haus{m-1}\res K_\eps}\bigr)\leq \gamma\qquad\forall t\in
(0,t_\eps).
$$
\end{proposition}
\begin{proof} We can assume with no loss of generality that
$1+\eps^2<2\pi$. For any $y \in K$, let $r_y>0$ be a radius such
that:
\begin{enumerate}
\item[-] $K\cap B_{r_y}(y)$ is contained inside a $C^1$ submanifold $S_y$;
\item[-] there exists an orthogonal transformation $Q_y:\R^m \to \R^m$ such that
$Q_y(S_y)$ is contained inside the graph of a Lipschitz function
$u_y:B_{r_y}^{m-1}\subset \R^{m-1}\to \R$;
\item[-] the Lipschitz constant of $u_y$ is bounded by $\eps$.
\end{enumerate}
By compactness, there exists a finite set of points $y_1,\ldots,y_N$
such that
$$
K\subset \bigcup_{i=1}^N B_{r_{y_i}}(y_i).
$$
Let us define the disjoints family of sets $A_1= K \cap
B_{r_{y_1}}(y_1)$, $A_i:= K \cap B_{r_{y_i}}(y_i) \setminus \left(
\cup_1^{i-1}A_{j}\right)$ for $i=2,\ldots,N$. For any given
$\eps>0$, we can find compact sets $E_i\subset A_i$ such that
$$
\sum_{i=1}^N \Haus{m-1}(A_i\setminus E_i) <\eps,\qquad \min_{1\leq
i\neq j\leq N} {\rm dist}(E_i,E_j)=:2\delta>0.
$$
Let us set $K_\eps:=\cup_{i=1}^N E_i$, and let $R>0$ be sufficiently
large so that $K_\eps\subset B_R$. Thanks to
Lemma~\ref{lemma:estimate one graph} below applied with
$\Gamma=Q_{y_i}(E_i)$ for $i=1\ldots,N$, since $G_m$ is invariant
under orthogonal transformations there exists $t_i>0$ such that
$$
\sqrt{t}{T_t^*\bigl(G_m\Haus{m-1}\res E_i}\bigr)\leq
\sqrt{\frac{1+\eps^2}{2\pi}} \Omega_{m,R}\left({\rm
dist}(\cdot,E_i)/\sqrt{t}\right) \gamma\qquad\forall t\in (0,t_i).
$$
This implies that, for $0<t<\min_i t_i$,
$$
\sqrt{t}{T_t^*\bigl(G_m\Haus{m-1}\res K_\eps\bigr)}\leq
\sqrt{\frac{1+\eps^2}{2\pi}} \sum_{i=1}^N \Omega_{m,R}\left({\rm
dist}(\cdot,E_i)/\sqrt{t}\right) \gamma.
$$
Recalling that ${\rm dist}(E_i,E_j)\geq 2\delta>0$ for $i \neq j$,
for all $x \in \R^m$ it holds ${\rm dist}(x,E_i)>\delta$ for all $i$
with at most one exception. Hence, since $\Omega_{m,R}\leq 1$
and $\Omega_{m,R}(s)\to 0$ as $s \to +\infty$, we
get
$$
\sqrt{\frac{1+\eps^2}{2\pi}} \sum_{i=1}^N  \Omega_{m,R}\left({\rm
dist}(\cdot,E_i)/\sqrt{t}\right) \leq \sqrt{\frac{1+\eps^2}{2\pi}}
\left( 1+ (N-1)  \Omega_{m,R}\left(\delta/\sqrt{t}\right)
\right)\leq 1
$$
for $t$ sufficiently small, which concludes the proof. \end{proof}

\begin{lemma}\label{lemma:estimate one graph}
Let $A\subset \R^{m-1}$ be a bounded Borel set, let $u:A\mapsto\R$
be a Lipschitz function with Lipschitz constant $\ell$, and let
$\Gamma:=\{(z,u(z)):\ z \in A\}$ be the graph of $u$. Assume that
$\Gamma\subset B_R$ for some $R>0$. Then, there exist a continuous
function $\Omega_{m,R}:[0,+\infty) \to [0,1]$,
depending only on $m$ and $R$, and $\bar t>0$, such that
$\Omega_{m,R}(s)\to 0$ as $s \to +\infty$, and
$$
\sqrt{t}{T_t^*\bigl(G_m\Haus{m-1}\res \Gamma\bigr)}\leq
\sqrt{\frac{1+\ell^2}{2\pi}} \Omega_{m,R}\left({\rm
dist}(x,\Gamma)/\sqrt{t}\right)\gamma\qquad\forall t\in (0,\bar t).
$$
\end{lemma}
\begin{proof}
Let us first observe that, given a test function $f:\R^m\to \R$, it
holds
\begin{eqnarray*}
\int_{\R^m} f \,dT_t^*\bigl(G_m\Haus{m-1}\res \Gamma\bigr)
&=&\int_\Gamma T_t f(y)G_m(y)\,d\Haus{m-1}(y)\\
&=& \int_{\R^m} f(x) \int_{\Gamma}\frac{e^{-\frac{|e^{-t}x|^2 -
2e^{-t}x\cdot y +|e^{-t}y|^2}{2(1-e^{-2t})}}}
{(1-e^{-2t})^{m/2}}G_m(y)\,d\Haus{m-1}(y)\,d\gamma(x).
\end{eqnarray*}
Hence, we have to show that, for any $x=(x',x_m) \in \R^{m-1}\times
\R$, the expression
$$
\sqrt{t}\int_{\Gamma}\frac{e^{-\frac{|e^{-t}x|^2 - 2e^{-t}x\cdot y
+|e^{-t}y|^2}{2(1-e^{-2t})}}}
{(1-e^{-2t})^{m/2}}G_m(y)\,d\Haus{m-1}(y)
=\frac{\sqrt{t}}{(2\pi)^{m/2}(1-e^{-2t})^{m/2}}\int_\Gamma
e^{-\frac{|e^{-t}x-y|^2}{2(1-e^{-2t})}}\,d\Haus{m-1}(y)
$$
is bounded by $\sqrt{\frac{1+\ell^2}{2\pi}} \Omega_{m,R}\left({\rm
dist}(x,\Gamma)/\sqrt{t}\right)$ for $t$ sufficiently small
(independent of $x$), with $\Omega_{m,R}$ as in the statement.

Thanks to the area formula and the bound on the Lipschitz constant,
we can write
\begin{eqnarray*}
&&\frac{\sqrt{t}}{(2\pi)^{m/2}(1-e^{-2t})^{m/2}}\int_\Gamma
e^{-\frac{|e^{-t}x-y|^2}{2(1-e^{-2t})}}\,d\Haus{m-1}(y)\\
&=&\frac{\sqrt{t}}{(2\pi)^{m/2}(1-e^{-2t})^{m/2}}\int_{A}
e^{-\frac{|e^{-t}x'-y'|^2}{2(1-e^{-2t})}}e^{-\frac{|e^{-t}x_m-u(y')|^2}{2(1-e^{-2t})}}
\sqrt{1+|\nabla u(y')|^2}\,dy'\\
&\leq&\frac{\sqrt{1+\ell^2}\sqrt{t}}{(2\pi)^{m/2}(1-e^{-2t})^{m/2}}\int_{A}
e^{-\frac{|e^{-t}x'-y'|^2}{2(1-e^{-2t})}}
e^{-\frac{|e^{-t}x_m-u(y')|^2}{2(1-e^{-2t})}}\,dy'.
\end{eqnarray*}
Now, since $t \leq 1-e^{-2t}$ for $t$ small, we can bound the above
expression by
\begin{equation}
\label{eq:bound}
\sqrt{\frac{1+\ell^2}{2\pi}}\frac{1}{(2\pi)^{(m-1)/2}(1-e^{-2t})^{(m-1)/2}}\int_{A}
e^{-\frac{|e^{-t}x'-y'|^2}{2(1-e^{-2t})}}e^{-\frac{|e^{-t}x_m-u(y')|^2}{2(1-e^{-2t})}}\,dy'.
\end{equation}
First of all we observe that, since
$$
\frac{1}{(2\pi)^{(m-1)/2}(1-e^{-2t})^{(m-1)/2}}\int_{A}
e^{-\frac{|e^{-t}x'-y'|^2}{2(1-e^{-2t})}}\,dy'=T_t \chi_A(x')\leq 1,
$$
the quantity in \eqref{eq:bound} is trivially bounded by $
(1+\ell^2)/(2\pi)$.

To show the existence of a function $\Omega_{m,R}$ as in the
statement of the lemma, we split the integral over $A$
into the one over $A\setminus B_{{\rm dist}(x,\Gamma)/2}(x')$, and
the one over $A\cap B_{{\rm dist}(x,\Gamma)/2}(x')$.

To estimate the first integral, we bound
$e^{-|e^{-t}x_m-u(y')|^2/[2(1-e^{-2t})]}$ by $1$. Moreover, we
observe that
\begin{eqnarray*}
T_t \chi_{A\setminus B_{{\rm dist}(x,\Gamma)/2}(x')}(x') &\leq
&\frac{1}{(2\pi)^{(m-1)/2}(1-e^{-2t})^{(m-1)/2}}\int_{\R^{m-1}\setminus
B_{{\rm dist}(x,\Gamma)/2}(x')}
e^{-\frac{|e^{-t}x'-y'|^2}{2(1-e^{-2t})}}\,dy'\\
&=&\frac{1}{(2\pi)^{(m-1)/2}}\int_{\R^{m-1}\setminus B_{{\rm
dist}(x,\Gamma)/[2\sqrt{1-e^{-2t}}]}}
e^{-\frac{|e^{-t}x'-x' - \sqrt{1-e^{-2t}}z'|^2}{2(1-e^{-2t})}}\,dz'\\
&=&\frac{1}{(2\pi)^{(m-1)/2}}\int_{\R^{m-1}\setminus B_{{\rm
dist}(x,\Gamma)/[2\sqrt{1-e^{-2t}}]}}
e^{-\frac{\left|z'+\sqrt{\frac{1-e^{-t}}{1+e^{-t}}}
x'\right|^2}{2}}\,dz'.
\end{eqnarray*}
We now remark that $-|a+b|^2 \leq -|a|^2/2 +|b|^2$ for all $a,b \in
\R^{m-1}$, $1-e^{-2t}\leq 2t$, and $\frac{1-e^{-t}}{1+e^{-t}} \leq
t$ for $t$ small. Hence, the above expression is bounded from above
by
$$
\frac{1}{(2\pi)^{(m-1)/2}}\int_{\R^{m-1}\setminus B_{{\rm
dist}(x,\Gamma)/(2\sqrt{2t})}} e^{-|z'|^2/4}e^{t|x'|^2/2}\,dz'.
$$
Since $\Gamma \subset B_R$ for some $R$, it holds $|x'| \leq |x|
\leq R+{\rm dist}(x,\Gamma)$, and so the above quantity can be
bounded from above by
\begin{eqnarray*}
&&\frac{1}{(2\pi)^{(m-1)/2}}e^{tR^2}e^{t{\rm
dist}(x,\Gamma)^2}\int_{\R^{m-1}\setminus B_{{\rm
dist}(x,\Gamma)/(2\sqrt{2t})}}
e^{-|z'|^2/4}\,dz'\\
&\leq&\frac{m\omega_m}{(2\pi)^{(m-1)/2}} e^{R^2}e^{{\rm
dist}(x,\Gamma)^2/100 t}\int_{{\rm
dist}(x,\Gamma)/(4\sqrt{t})}^{\infty} e^{-\tau^2/4}\tau^{m-1}\,d\tau
\end{eqnarray*}
for $t$ small (here $\omega_m$ denotes the measure of the unit ball in $\R^m$).

To control the second integral over
$A\cap B_{{\rm dist}(x,\Gamma)/2}(x')$, we bound $T_t \chi_{A\cap
B_{{\rm dist}(x,\Gamma)/2}(x')}(x')$ by $1$ and we estimate from
above, uniformly for $y'\in B_{{\rm dist}(x,\Gamma)/2}(x')$, the
quantity
$$
e^{-\frac{|e^{-t}x_m-u(y')|^2}{2(1-e^{-2t})}}.
$$
We proceed as follows: for $y' \in B_{{\rm dist}(x,\Gamma)/2}(x')$,
by the definition of ${\rm dist}(x,\Gamma)$, we have
$$
4|x'-y'|^2\leq {\rm dist}(x,\Gamma)^2 \leq |x'-y'|^2+|x_m -
u(y')|^2,
$$
which implies $3|x'-y'|^2\leq |x_m - u(y')|^2$, and so ${\rm
dist}(x,\Gamma)^2 \leq 4|x_m - u(y')|^2/3.$ Thus, using again the
estimate $-|a-b|^2 \leq -|a|^2/2 +|b|^2$, for $t$ small enough we
obtain
$$
e^{-\frac{|e^{-t}x_m-u(y')|^2}{2(1-e^{-2t})}} \leq e^{-\frac{|x_m -
u(y')|^2}{4(1-e^{-2t}))}}e^{\frac{(1-e^{-t})^2|x_m|^2}{(1-e^{-2t})}}
\leq e^{-{\rm dist}(x,\Gamma)^2/(16t)}e^{t|x_m|^2}.
$$
Since $|x_m| \leq |x| \leq R+{\rm dist}(x,\Gamma)$, we conclude that
$$
e^{-\frac{|e^{-t}x_m-u(y')|^2}{2(1-e^{-2t})}} \leq e^{R^2}e^{-{\rm
dist}(x,\Gamma)^2/(20t)} \qquad \forall y' \in B_{{\rm
dist}(x,\Gamma)/2}(x')
$$
for $t$ small enough.

Hence, it suffices to define
$$
\Omega_{m,R}(s):=\min\left\{1,
\frac{m\omega_m}{(2\pi)^{(m-1)/2}}e^{R^2}e^{s^2/100}\int_{s/4}^{\infty}
e^{-\tau^2/4}\tau^{m-1}\,d\tau+e^{R^2} e^{-s^2/20}\right\}
$$
(recall that $\int_{s/4}^{\infty} e^{-\tau^2/4}\tau^{m-1}\,d\tau\sim
c_me^{-s^2/64}s^{m-2}$ as $s \to +\infty$) to conclude the proof.
\end{proof}

The next lemma is stated with outer integrals $\int_Y^*$; this
suffices for our purposes and avoids the difficulty of proving that
the measures $\sigma_y$ we will dealing with have a measurable
dependence w.r.t. $y$.

\begin{lemma}\label{lrectifiable}
Let $(Y,{\cal F},\mu)$ be a probability space and, for $t>0$ and
$y\in Y$, let $g_{t,y}:X\to [0,1]$ be Borel maps. Assume also that:
\begin{itemize}
\item[(a)]$\{\sigma_y\}_{y\in Y}$ are
positive finite Borel measures in $X$, with
$\int_Y^*\sigma_y(X)\,d\mu(y)$ finite;
\item[(b)] $\sigma_y=G_m\Haus{m-1}\res\Gamma_y$ for $\mu$-a.e. $y$,
with $\Gamma_y$ countably $\Haus{m-1}$-rectifiable.
\end{itemize}
Then
\begin{equation}\label{hino7}
\limsup_{t\downarrow 0}\int_Y^*\int_X T_t
g_{t,y}(x)\,d\sigma_y(x)d\mu(y)\leq\limsup_{t\downarrow 0}
\frac{1}{\sqrt{t}}\int_Y^*\int_X g_{t,y}(x)\,d\gamma(x) d\mu(y).
\end{equation}
\end{lemma}
\begin{proof} We prove first the lemma under the stronger
assumption that, for $\mu$-a.e. $y\in Y$, there exists $t_y>0$ such
that
$$
T_t^*\sigma_y\leq \frac{1}{\sqrt{t}}\gamma\qquad\forall t\in
(0,t_y).
$$
Fix $\eps>0$ small, and set $Y_{\eps}:=\{y\in Y:\ t_y>\delta\}$, where
$\delta=\delta(\eps)>0$ is chosen sufficiently small in such a way
that $\int_{Y_\eps}^*\int_X T_t g_{t,y}\,d\sigma_y d\mu(y)+\eps\geq
\int_Y^*\int_X T_t g_{t,y}\,d\sigma_y d\mu(y)$ (this is possible, by
the continuity properties of the upper integral). For $t\in
(0,\delta)$ we estimate the integrals in \eqref{hino7} with $Y_\eps$
in place of $Y$:
$$
\int_{Y_\eps}^*\int_X T_t g_{t,y}\,d\sigma_y d\mu(y)=
\int_{Y_\eps}^*\int_X g_{t,y}\,dT_t^*\sigma_yd\mu(y)\leq
\frac{1}{\sqrt{t}}\int_Y^*\int_X g_{t,y}\,d\gamma d\mu(y).
$$
Hence, letting $t\downarrow 0$ yields \eqref{hino7} with an extra
summand $\eps$ in the right hand side. Since $\eps$ is arbitrary we
conclude.

Finally, in the general case when $\Gamma_y$ is countably
$\Haus{m-1}$-rectifiable we can find for any $\eps>0$ sets
$\Gamma_y'\subset\Gamma_y$ contained in the union of finitely many
hypersurfaces such that $\sigma_y(\Gamma_y\setminus
\Gamma_y')<\eps/2$ and then, thanks to Proposition~\ref{findim2},
sets $\Gamma_y''\subset\Gamma_y'$ with
$\sigma_y(\Gamma_y'\setminus\Gamma_y'')<\eps/2$ in such a way that
the estimate \eqref{hino7} holds when $\sigma_y$ is replaced by
$G_m\Haus{m-1}\res\Gamma_y''$. Since $T_tg_t\leq 1$ we
can let $\eps\downarrow 0$ to obtain \eqref{hino7}.\end{proof}

In the proof of Theorem~\ref{tessbou2} we need a Poincar\'e
inequality involving capacities. Recall that the $1$-dimensional
capacity $c_1(G)$ of a Borel set $G$ can be defined as:
$$
c_1(G):=\inf\left\{|Du|(\R^m):\ u\in L^{m/(m-1)}(\R^m),\,\,G\subset
{\rm int}(\{u\geq 1\})\right\}
$$
(see \cite[\S5.12]{Ziemer}; other equivalent definitions involve the
Bessel capacity). The following result is known (see for instance
\cite[Theorem~5.13.3]{Ziemer}) but we reproduce it for the reader's
convenience in the simplified case when $v$ is continuous.

\begin{lemma}\label{lpoincare2}
Let $v\in W^{1,1}(B_{r})\cap C(B_r)$ and let $G\subset B_r$ be a
Borel set with $c_1(G)>0$. Then, for some dimensional constant
$\kappa$, it holds
$$
\frac{1}{\omega_mr^m}\int_{B_r}|v|\,dx\leq
\frac{\kappa}{c_1(G)}\int_{B_r}|\nabla v|\,dx
$$
whenever $v$ vanishes $c_1$-a.e. on $G$.
\end{lemma}
\begin{proof} By a scaling argument, suffices to consider
the case $r=1$. By a truncation argument (i.e., first considering
the positive and negative parts and then replacing $v$ by
$\min\{v,n\}$ with $n \in\N$) we can also assume that $v$ is nonnegative and
bounded. By homogeneity of both sides, suffices to consider the case
$0\leq v\leq 1$. In this case the statement follows by applying the
inequality
\begin{equation}\label{hino15}
\Leb{m}(B_1\setminus E)\leq\frac{\kappa}{c_1(G)}
|D\chi_E|(B_1)\qquad\text{whenever $E$ is open and $G\subset E$}
\end{equation}
with $E=\{v<t\}$, $t\in (0,1)$, and then integrating both sides with respect to $t$ and using the coarea formula. Hence,
we are led to the proof of \eqref{hino15}. Now, if
$\Leb{m}(E)\geq\omega_m/2$ we can apply the relative isoperimetric
inequality in $B_1$ to get
$$
\Leb{m}(B_1\setminus E)\leq c_m|D\chi_E|(B_1)\leq
\frac{\kappa}{c_1(G)}|D\chi_E|(B_1)
$$
provided we choose $\kappa$ so large that $\kappa\geq c_1(B_1) c_m$ (observe that $c_1(G) \leq c_1(B_1)$).
On the other hand, if $\Leb{m}(E)\leq\omega_m/2$ then we estimate
$\Leb{m}(B_1\setminus E)$ from above with $\omega_m$ and it suffices to show
that $|D\chi_E|(B_1)\geq c_1(G)\omega_m/\kappa$ for
$\kappa=\kappa(m)$ large enough. In this case we can find a
compactly supported $BV$ function $u$ coinciding with $\chi_E$ on
$B_1$ with
$$|Du|(\R^m)\leq c_m' \bigl(|D\chi_E|(B_1)+\Leb{m}(E\cap
B_1)\bigr)\leq c_m'(1+c_m)|D\chi_E|(B_1)
$$
(see for instance \cite[Proposition~3.21]{afp} for the existence of
a continuous linear extension operator from $BV(B_1)$ to
$BV(\R^m)$). It follows that $c_1(G)\leq c_m'(1+c_m)|D\chi_E|(B_1)$,
so suffices to take $\kappa$ such that $\kappa/\omega_m\geq
c_m'(1+c_m)$.\end{proof}

In the sequel we shall extensively use the following identity
between null sets w.r.t. $c_1$ and null sets w.r.t. to codimension
one Hausdorff measure, see for instance \cite[Lemma~5.12.3]{Ziemer}:
\begin{equation}\label{hino14}
c_1(G)=0\qquad\Longleftrightarrow\qquad\Haus{m-1}(G)=0.
\end{equation}

\begin{lemma}\label{ldensity} Let $G\subset\R^m$ be a Borel set. Then
$$
\limsup_{r\downarrow 0}\frac{c_1(G\cap B_r(x))}{r^{m-1}}>0
\qquad\text{for $c_1$-a.e. $x\in G$.}
$$
\end{lemma}
\begin{proof} Let $L\subset G$ be the Borel set of points
where the limsup is null and assume by contradiction that
$c_1(L)>0$. Then \eqref{hino14} yields $\Haus{m-1}(L)>0$ as well and
we can find, thanks to \cite{Besicovitch}, a compact subset $L'$
with $0<\Haus{m-1}(L')<\infty$. We will prove that
\begin{equation}\label{hino12}
\liminf_{r\downarrow 0}\frac{c_1(L'\cap
\overline{B}_r(x))}{\Haus{m-1}(L'\cap
\overline{B}_r(x))}>0\qquad\text{for $\Haus{m-1}$-a.e. $x\in L'$.}
\end{equation}
Combining this information with the well-know fact (see for instance
\cite[(2.43)]{afp})
\begin{equation}\label{hino13}
\limsup_{r\downarrow 0}\frac{\Haus{m-1}(L'\cap
\overline{B}_r(x))}{r^{m-1}}>0 \qquad\text{for $\Haus{m-1}$-a.e.
$x\in L'$,}
\end{equation}
we obtain
$$
\limsup_{r\downarrow 0}\frac{c_1(L'\cap
\overline{B}_r(x))}{r^{m-1}}>0 \qquad\text{for $\Haus{m-1}$-a.e.
$x\in L'$,}
$$
in contradiction with the inclusion $L'\subset L$ and the fact that
$\Haus{m-1}(L')>0$.

To conclude the proof, we check \eqref{hino12}. Let $L''\subset L'$
be the Borel set of points where the liminf in \eqref{hino12} is
null; for all $\eps>0$ we can find, thanks to Vitali covering theorem,
a disjoint cover of $\Haus{m-1}$-almost all of $L''$ by disjoint
closed balls $\{\overline{B}_{r_i}(x_i)\}_{i\in I}$ satisfying
$c_1(L'\cap
\overline{B}_{r_i}(x_i))\leq\eps\Haus{m-1}(L'\cap\overline{B}_{r_i}(x_i))$.
Thanks to \eqref{hino14} the balls cover also $c_1$-almost all of $L''$, so the countable
subadditivity of capacity yields $c_1(L'')\leq\eps\Haus{m-1}(L')$. Since
$\eps$ is arbitrary we conclude that $c_1(L'')=0$, whence
$\Haus{m-1}(L'')=0$ by \eqref{hino14}.
\end{proof}

\begin{proposition}\label{pfederer}
Let $(u_n)\subset W^{1,1}(X,\gamma)\cap C(X)$ be convergent in
$L^1(X,\gamma)$ to $\chi_E$, with $E$ of finite perimeter, and
satisfying
\begin{equation}\label{hino23}
\limsup_{n\to\infty}\int_X|\nabla u_n|\,d\gamma\leq
|D_\gamma\chi_E|(X).
\end{equation}
Then
$$
L:=\left\{x:\ \lim_{n\to\infty} u_n(x)=\frac{1}{2}\right\}
$$
is contained, up to $\Haus{m-1}$-negligible sets, in the essential
boundary of $E$.
\end{proposition}
\begin{proof} Possibly passing to the smaller sets
$$
L \cap \biggl(\bigcup_{n=m}^\infty\left\{x\in X:\ |u_n(x)-\frac12|\leq\frac
14\right\}\biggr)
$$
which monotonically converge to $L$ as $m\to\infty$, we can assume
with no loss of generality that $|u_n-1/2|\leq 1/4$ on $L$.

Let us prove, first, that \eqref{hino23} yields the weak$^*$ convergence
in the duality with $C_b(X)$ of $|\nabla u_n|\gamma$ to
$|D_\gamma\chi_E|$. It suffices to apply the lower semicontinuity of
the total variation in open sets (see Proposition \ref{prop:sci}) to get
$$
\liminf_{n\to\infty}\int_A|\nabla
u_n|\,d\gamma\geq|D_\gamma\chi_E|(A)\qquad \text{for all $A\subset
X$ open}
$$
and then to apply \cite[Proposition~1.80]{afp}.

Denoting by $E^1$ the set of density points of $E$, it suffices to
show that $c_1(L\cap E^1)=0$; indeed, the same property with the
complement of $E$ and $1-u_n$ gives $c_1(L\cap E^0)=0$, where $E$ is
the set of rarefaction points of $E$, and since $E^0\cup E^1$ is the
complement of the essential boundary of $E$ we conclude thanks to \eqref{hino14}.

We now assume by contradiction that $G:=L\cap E^1$ has strictly
positive capacity. Since $|D\chi_E|(B_r(y))=o(r^{m-1})$ for
$\Haus{m-1}$-a.e. $y\in E^1$ and thanks to Lemma~\ref{ldensity}, we find a point
$x\in G$ and radii $r_i\downarrow 0$ such that $\lim_ic_1(G\cap
B_{r_i}(x))/r_i^{m-1}>0$ and
$|D\chi_E|(B_{r_i}(x))=o(r_i^{m-1})$. Let $\phi:[0,1]\to [0,1]$ be the
piecewise affine function identically equal to $1/2$ on $[1/4,3/4]$
and with derivative equal to $2$ on $(0,1/4)\cup (3/4,1)$. Since
$\phi\circ u_n$ are identically equal to $1/2$ on $L\supset G$, we
can apply Lemma~\ref{lpoincare2} to $1/2-\phi\circ u_n$ in the ball
$B_{r_i}(x)$ to get
$$
r_i^{-m}\int_{B_{r_i}(x)}|\phi\circ u_n-\frac12|\,dy\leq
\frac{2\kappa \omega_m}{c_1(G\cap
B_{r_i}(x_i))}\int_{B_{r_i}(x)}|\nabla u_n|\,dy.
$$
Since $\phi(0)=0$ and $\phi(1)=1$, passing to the limit as
$n\to\infty$ and using the weak$^*$ convergence of $|\nabla u_n|\gamma$
to $|D_\gamma\chi_E|$ yields
$$
r_i^{-m}\int_{B_{r_i}(x)}|\chi_E-\frac12|\,dy\leq
\frac{2\kappa\omega_m}{c_1(G\cap
B_{r_i}(x_i))}\int_{B_{r_i}(x)}\frac{1}{G_m}\,d|D_\gamma\chi_E|.
$$
Since $r_i^{m-1}/c_1(G\cap B_{r_i}(x_i))$ is uniformly bounded as
$i\to\infty$ and $|D\chi_E|(B_{r_i}(x))=o(r_i^{m-1})$ we conclude that
$$
r_i^{-m}\int_{B_{r_i}(x)}|\chi_E-\frac12|\,dy\to 0 \qquad \text{as }r_i\downarrow 0,
$$
contradicting the fact that $x\in
E^1$.\end{proof}

\section{Convergence of $T_t\chi_E$ to $1/2$ on $\partial^*E$}

In this section we shall prove Theorem~\ref{tessbou1}. By a
well-known convergence criterion in $L^2$, the stated convergence
will be a consequence of the weak$^*$ convergence of $T_t\chi_E$ to
$1/2$ in $L^\infty(X,|D_\gamma\chi_E|)$, that we shall prove in
Proposition~\ref{pweakstar}, and the following apriori estimate (see
also Remark~\ref{ralessio}):

\begin{proposition} \label{papriori14}
For any set $E$ with finite perimeter in $(X,\gamma)$ it holds
\begin{equation}\label{hino2}
\limsup_{t\downarrow 0}\int_X |T_t\chi_E|^2\,d|D_\gamma\chi_E|\leq
\frac{1}{4}|D_\gamma\chi_E|(X).
\end{equation}
\end{proposition}

\begin{proof} In this proof we shall use the simpler
notation
$$
T_tf(x)=\int_F f(y)\rho^X_t(x,dy)
$$
for the action of the OU semigroup. Comparing with Mehler's formula
\eqref{mehler}, we see that the measure $\rho_t^X(x,\cdot)$ is
nothing but the law of $y\mapsto e^{-t}x+\sqrt{1-e^{-2t}}y$ under
$\gamma$ (not absolutely continuous w.r.t. $\gamma$ if $t>0$ and $X$
is infinite-dimensional).

Let $f_t=T_t\chi_E$ and write, as in \eqref{factorization},
$$
f_t(z,y)=\int_Y\int_F\chi_{E_{y'}}(z')\rho^F_t(z,dz')\rho^Y_t(y,dy')
$$
where $H=F\oplus F^\perp$ is an orthogonal decomposition of $H$,
$F\subset\tilde{H}$ is finite-dimensional, $X=F\oplus Y$ and
$\gamma=\gamma_F\otimes\gamma_Y$ are the corresponding
decompositions of $X$ and $\gamma$ and $E_y=\{z\in F: (z,y)\in E\}$.
Then H\"older's inequality yields
\begin{equation}\label{hino3}
f_t^2(z,y)\leq\int_Y\biggl(\int_{E_{y'}}\rho^F_t(z,dz')\biggr)^2
\,\rho^Y_t(y,dy'),
\end{equation}
so that it suffices to estimate from above the upper limits of the
integrals
\begin{equation}\label{hino4}
\int_X\biggl[ \int_Y\biggl(\int_{E_{y'}}\rho^F_t(z,dz')\biggr)^2
\,\rho^Y_t(y,dy')\biggr]\,d|D_\gamma\chi_E|(x)
\end{equation}
as $t\downarrow 0$, with $|D_\gamma\chi_E|(X)/4$. First of all, we
notice that the quantity in square parentheses is less than 1;
hence, since \eqref{hino40} ensures that the measures in $X$
$$
|D_{\gamma_F}\chi_{E_y}|(dz)\otimes\gamma_Y(dy)
$$
monotonically converge to $|D_\gamma\chi_E|$ as $F\uparrow H$ (more
precisely, as $F$ increases to a vector space dense in $H$), it
suffices to estimate with $|D_\gamma\chi_E|(X)/4$ the upper limit as
$t\downarrow 0$ of the integrals
\begin{equation}\label{hino5}
\int_Y\int_F\biggl[
\int_Y\biggl(\int_{E_{y'}}\rho^F_t(z,dz')\biggr)^2
\,\rho^Y_t(y,dy')\biggr]\,d|D_{\gamma_F}\chi_{E_y}|(z) d\gamma_Y(y).
\end{equation}
Now, if in \eqref{hino5} we replace the innermost integral on
$E_{y'}$ with an integral on $E_y$, thanks to Fatou's lemma and
Proposition~\ref{findim1} (observe that
$\int_{E_{y}}\rho^F_t(z,dz')\leq 1$) we get immediately
\begin{eqnarray*}
&&\limsup_{t\downarrow 0} \int_Y\int_F
\biggl(\int_{E_{y}}\rho^F_t(z,dz')\biggr)^2\,d|D_{\gamma_F}\chi_{E_y}|(z) d\gamma_Y(y)\\
&\leq&\int_Y\int_F\limsup_{t\downarrow 0}
\biggl(\int_{E_{y}}\rho^F_t(z,dz')\biggr)^2\,d|D_{\gamma_F}\chi_{E_y}|(z) d\gamma_Y(y)\\
&\leq&
\frac{1}{4}\int_Y|D_{\gamma_F}\chi_{E_y}|(F)\,d\gamma_Y(y).
\end{eqnarray*}
Since the quantity above is
less than $|D_\gamma\chi_E|(X)/4$, we are led to show that the
$\limsup$ as $t\downarrow 0$ of the expressions
$$
\int_Y\int_F\int_Y\biggl|
\biggl(\int_{E_{y'}}\rho^F_t(z,dz')\biggr)^2-
\biggl(\int_{E_y}\rho^F_t(z,dz')\biggr)^2\biggr|
\,\rho^Y_t(y,dy')\,d|D_{\gamma_F}\chi_{E_y}|(z) d\gamma_Y(y)
$$
can be made arbitrarily small, choosing $F$ large enough. To this
aim, bounding the difference of the squared integrals with twice
their difference, using again that $\int_{E_{y}}\rho^F_t(z,dz')\leq 1$ it suffices to estimate the simpler expressions
\begin{equation}\label{hino6}
\int_Y\int_F\int_Y\biggl|\biggl(\int_{E_{y'}}\rho^F_t(z,dz')-
\int_{E_y}\rho^F_t(z,dz')\biggr)\biggr|\,\rho^Y_t(y,dy')\,d|D_{\gamma_F}\chi_{E_y}|(z)
d\gamma_Y(y).
\end{equation}
We can now estimate \eqref{hino6} from above with
$$
\int_Y\int_F T^F_t g_{t,y}(z) \,d|D_{\gamma_F}\chi_{E_y}|(z)
d\gamma_Y(y),
$$
where $T^F_t$ denotes the OU semigroup in $(F,\gamma_F)$ and
$$
g_{t,y}(z):=\int_Y|\chi_{E_{y'}}(z)-\chi_{E_y}(z)|\rho^Y_t(y,dy').
$$
Keeping $y$ fixed, by applying Lemma~\ref{lrectifiable} with
$\sigma_y=|D_{\gamma_F}\chi_{E_y}|$ we get that the limsup as
$t\downarrow 0$ of the expression in \eqref{hino6} is bounded above
by
\begin{equation}\label{hino8}
\limsup_{t\downarrow 0}\frac{1}{\sqrt{t}}\int_Y\int_X
g_{t,y}(z)\,d\gamma_F(z)d\gamma_Y(y).
\end{equation}
Since we can also write
$g_{t,y}(z)=\int_Y\bigl|\chi_{E_{z}}(y)-\chi_{E_z}(y')\bigr|\rho^Y_t(y,dy')$,
by \eqref{poincarestr} we get
$$
\int_Yg_{t,y}(z)\,d\gamma_Y(y)= \int_Y\int_Y
|\chi_{E_z}(y)-\chi_{E_z}(y')|\rho^Y_t(y,dy')\,d\gamma_Y(y) \leq c_t
|D_{\gamma_Y}\chi_{E_z}|(Y),
$$
so that an integration w.r.t. $z$ and Fubini's theorem give that the
$\limsup$ in \eqref{hino8} is bounded above by (taking also into
account that $c_t\sim 2\sqrt{t/\pi}$)
$$
\frac{2}{\sqrt{\pi}}\int_F|D_{\gamma_Y}\chi_{E_z}|(Y)\,d\gamma_F(z).
$$
But, according to \eqref{hino41}, we can represent the expression
above as
$$
\frac{2}{\sqrt{\pi}}\int_X|\pi_{F^\perp}(\nu_E)|\,d|D_\gamma\chi_E|.
$$
Since $|\pi_{F^\perp}(\nu_E)|\downarrow 0$ as $F$ increases to a
dense subspace of $H$, we conclude. \end{proof}

\begin{remark}\label{ralessio} {\rm
In the previous proof we used that the statement is true in finite
dimensions, see Proposition~\ref{findim1}. But actually
Proposition ~\ref{findim1} provides also a stronger information, and the proof
above could be slightly modified to obtain directly
Theorem~\ref{tessbou1} from this stronger information. However, we
prefer to emphasize a softer and surely more elementary proof of the
weak$^*$ convergence of $T_t$. Indeed, we believe that the softer
argument below (based just on the product rule \eqref{product} and
some elementary arguments) has an interest in his own. In
particular, a variant of this argument allows to prove that
$|D_\gamma\chi_E|$ is also concentrated on a kind of reduced
boundary (see the Appendix).}
\end{remark}

\begin{proposition}\label{pweakstar} As $t\downarrow 0$, $T_t\chi_E$ weak$^*$
converge to $1/2$ in $L^\infty(X,|D_\gamma\chi_E|)$.
\end{proposition}
\begin{proof} Let $t_i\downarrow 0$ be such that
$f_i:=T_{t_i}\chi_E$ weak$^*$ converge to some function $f$ as
$i\to\infty$. It suffices to show that $f\geq 1/2$ up to
$|D_\gamma\chi_E|$-negligible sets. Indeed, the same property
applied to $X\setminus E$ yields $1-f\geq 1/2$ up to
$|D_\gamma\chi_{X\setminus E}|$-negligible sets, and since the
surface measures of $E$ and $X\setminus E$ are the same we obtain
that $f=1/2$ in $L^\infty(X,|D_\gamma\chi_E|)$. Since $T_t\chi_E$ is
uniformly bounded in $L^\infty(X,|D_\gamma\chi_E|)$, from the
arbitrariness of $(t_i)$ the stated convergence property as
$t\downarrow 0$ follows.

By approximation, it suffices to show that
\begin{equation}\label{hino1}
2\int_Af\,d|D_\gamma\chi_E|\geq |D_\gamma\chi_E|(A)
\end{equation}
for any open set $A\subset X$; by inner approximation with smaller
open sets whose boundary is $|D_\gamma\chi_E|$-negligible, we can
also assume in the proof of \eqref{hino1} that
$|D_\gamma\chi_E|(\partial A)=0$. We use the product rule
\eqref{product} to obtain
$$
D_\gamma(f_i\chi_E)=f_i D_\gamma\chi_E+\chi_E\nabla f_i\gamma.
$$
Then, we use the relations $\nabla T_tv=e^{-t}T_t^*D_\gamma v$ (see
Proposition~\ref{pammiss1}(b)) and $|\nabla T_tv|\leq
e^{-t}T_t^*|D_\gamma v|$ with $v=\chi_E$ and $t=t_i$ to get
$$
|D_\gamma(f_i\chi_E)|\leq f_i|D_\gamma\chi_E|+
T_{t_i}^*|D_\gamma\chi_E|.
$$
Let us now evaluate both measures on $A$:
$$
|D_\gamma(f_i\chi_E)|(A)\leq \int_A f_i\,d|D_\gamma\chi_E|+ \int_X
T_{t_i}\chi_{A\cap E}\,d|D_\gamma\chi_E|.
$$
Since $T_{t_i}\chi_{A\cap E}\leq\min\{f_i,T_{t_i}\chi_A\}$ we can
further estimate
$$
|D_\gamma(f_i\chi_E)|(A)\leq 2\int_A f_i\,d|D_\gamma\chi_E|+
\int_{X\setminus A} T_{t_i}\chi_A\,d|D_\gamma\chi_E|.
$$
Finally, since $f_i\chi_E\to\chi_E$ in $L^1(X,\gamma)$, it suffices to use the fact that $T_t\chi_A\to 0$ pointwise on
$X\setminus\overline A$ and the lower semicontinuity of the total
variation in open sets (see Proposition \ref{prop:sci}) to get \eqref{hino1}.\end{proof}

\section{Representation of the perimeter measure}

In this section we shall prove Theorem~\ref{tessbou2}. We fix an
orthogonal decomposition $X=F\oplus F^\perp$ of $H$, with $F\subset
\tilde{H}$ finite-dimensional, and denote by $X=F\oplus Y$ the
corresponding decomposition of $X$. We define $E_y$, $y\in Y$, as in
\eqref{hino16} and, correspondingly, the essential boundary
$\partial^*_F E$ as in \eqref{hino17}.

Our main goal will be to show that the set $E^{1/2}$ (as defined in
Definition~\ref{disoessbou}), namely
$$
\left\{x\in X:\ \lim_{i\to\infty}
T_{t_i}\chi_E(x)=\frac{1}{2}\right\}
$$
is contained in $\partial^*_FE$ up to $\Haus{\infty-1}_F$-negligible
sets, i.e.,
\begin{equation}\label{hino18}
\Haus{\infty-1}_F(E^{1/2}\setminus\partial^*_FE)=0.
\end{equation}

\begin{proof}[Proof of \eqref{hino18}.] Let
$f_{i,y}(z)=T_{t_i}\chi_E(z,y)$. Since $\sum_i\sqrt{t_i}<\infty$ we
can use the estimates
$$
\int_Y\sum_i\int_F|f_{i,y}-\chi_{E_y}|\,d\gamma_F\,d\gamma_Y(y)=
\sum_i\int_X
|T_{t_i}\chi_E-\chi_E|\,d\gamma\leq|D_\gamma\chi_E|(X)\sum_i
c_{t_i},
$$
with $c_t$ as in Lemma~\ref{lpoincare}, to obtain that
$f_{i,y}\to\chi_{E_y}$ in $L^1(\gamma_F)$ for $\gamma_Y$-a.e. $y\in
Y$. Our first task will be to show the existence of a subsequence
$t_{i(j)}$ such that
\begin{equation}\label{hino22}
\lim_{j\to\infty}\int_F|\nabla_F f_{i(j),y}|\,d\gamma_F=
|D_{\gamma_F}\chi_{E_y}|(F)\qquad\text{for $\gamma_Y$-a.e. $y\in
Y$.}
\end{equation}
To this aim, we first show that
\begin{equation}\label{hino25}
\int_Y\biggl(\int_F|\nabla_F
f_{i,y}|\,d\gamma_F\biggr)\,d\gamma_Y\leq \int_Y
|D_{\gamma_F}\chi_{E_y}|(F)\,d\gamma_Y.
\end{equation}
In order to prove \eqref{hino25} we use
Proposition~\ref{pammiss1}(b) to get $|\nabla_F f_i|\gamma\leq
T_{t_i}^*|\pi_F(D_\gamma\chi_E)|$, hence
$$
\int_X|\nabla_F f_{i,y}|\,d\gamma\leq |\pi_F(D_\gamma\chi_E)|(X)
$$
and using \eqref{hino40} we conclude that \eqref{hino25} holds.

 Condition \eqref{hino22} now follows by the
$L^1(Y,\gamma_Y)$ convergence of $\int_F|\nabla_F
f_{i,y}|\,d\gamma_F$ to $|D_{\gamma_F}\chi_{E_y}|(F)$; in turn,
applying a convergence criterion (see for instance
\cite[Exercise~1.19]{afp}) this follows by the $\liminf$ inequality
$$
\liminf_{i\to\infty}\int_F|\nabla_F f_{i,y}|\,d\gamma_F\geq
|D_{\gamma_F}\chi_{E_y}|(F)\qquad\text{for $\gamma_Y$-a.e. $y\in
Y$.}
$$
(a consequence of the lower semicontinuity of total variation)
together with convergence of the $L^1$ norms ensured by
\eqref{hino25}.

Now, we fix $y$ such that all functions $f_{i,y}$ are continuous and
both conditions
$$
\lim_{i\to\infty}\int_F|f_{i,y}-\chi_{E_y}|\,d\gamma_F=0,\qquad
\lim_{j\to\infty}\int_F|\nabla_F f_{i(j),y}|\,d\gamma_F=
|D_{\gamma_F}\chi_{E_y}|(F)
$$
hold and apply Proposition~\ref{pfederer} to obtain that the $y$
section of $E^{1/2}$, contained in
$$
\left\{z\in F:\ \lim_{j\to\infty} f_{i(j),y}(z)=\frac{1}{2}\right\}
$$
is also contained, up to $\Haus{m-1}$-negligible sets, in
$\partial^* E_y$. Since Proposition~\ref{pbogaregu} and
\eqref{hino22} ensure that the set of exceptional $y$'s is
$\gamma_Y$-negligible, the definition of $\Haus{\infty-1}_F$ yields
\eqref{hino18}.\end{proof}

Having achieved \eqref{hino18} we can now prove
Theorem~\ref{tessbou2}. To this aim, we fix a nondecreasing family
${\cal F}=\{F_1,F_2,\ldots\}$ of finite-dimensional subspaces of
$\tilde{H}$ whose union is dense in $H$ and, using \eqref{hino18} in
conjunction with \eqref{hino19}, for $n\leq m$ we get
$$
\Haus{\infty-1}_{F_n}(E^{1/2}\setminus \bigcap_{i=m}^\infty
\partial^*_{F_i}E)=0.
$$
Letting $m\to \infty$ it follows that
$\Haus{\infty-1}_{F_n}(E^{1/2}\setminus\partial^*_{{\cal F}}E)=0$,
and since $n$ is arbitrary this proves that
\begin{equation}\label{hino20}
\Haus{\infty-1}_{{\cal F}}(E^{1/2}\setminus\partial^*_{{\cal
F}}E)=0.
\end{equation}
Now, we know that $|D_\gamma\chi_E|=\Haus{\infty-1}_{{\cal
F}}\res\partial^*_{{\cal F}}E$, hence evaluating both measures on
$\partial^*_{{\cal F}}E\setminus E^{1/2}$ and using the fact that
$|D_\gamma\chi_E|$ is concentrated on $E^{1/2}$ we get
\begin{equation}\label{hino21}
\Haus{\infty-1}_{{\cal F}}(\partial^*_{{\cal F}}E\setminus
E^{1/2})=0.
\end{equation}
The combination of \eqref{hino20} and \eqref{hino21} gives
$$
|D_\gamma\chi_E|=\Haus{\infty-1}_{{\cal F}}\res E^{1/2}.
$$
But, since ${\cal F}$ is arbitrary, this yields that $E^{1/2}$ has
finite $\Haus{\infty-1}$-measure and \eqref{hino31}, concluding the
proof.

\section{Derivative of the union of disjoint sets}
In this section we prove Corollary \ref{cor:product}.
Let us remark that, although the result is standard in finite dimensions and could be proved in different ways (for instance, using De Giorgi's rectifiability theorem),
the argument below is very elementary.
Although the proof is more or less the same as the one in \cite[Lemma 2.2]{fmp} (where the authors are dealing with the classical notion of perimeter in $\R^m$),
we believe it is worth to repeat the argument for reader's convenience, and for underlying the importance of the fact that in our representation formula \eqref{hino31}
the measure $\Haus{\infty -1}$ is universal.

\begin{proof}[Proof of Corollary \ref{cor:product}]
The fact that $E\cup F$ has finite perimeter follows immediately from the definition.

Since the sets $(E\cup F)^{1/2}$, $E^{1/2}$, $F^{1/2}$ are $\Haus{\infty -1}$-uniquely determined,
we can assume that they all have been defined using the same sequence $(t_i)$.

As $\gamma(E\cap F)=0$ we have $\chi_{E\cup F}=\chi_E+\chi_F$, so that by \eqref{hino31}
\begin{eqnarray}\nonumber
\nu_{E\cup F}\Haus{\infty-1}\res(E\cup F)^{1/2}&=&D_\gamma\chi_{E\cup F}=D_\gamma\chi_E+D_\gamma\chi_F\\
&=&\nu_E \Haus{\infty-1}\res E^{1/2}+\nu_F \Haus{\infty-1}\res F^{1/2}.
\label{dim lemma unione}
\end{eqnarray}
Since $E^{1/2}\cap F^{1/2}\subset \left\{x\in X:\
\lim_{i\to\infty}T_{t_i}\chi_{E\cup F}(x)=1\right\}$ we have
\begin{equation}
\label{eq:union}
(E\cup F)^{1/2}\cap E^{1/2}\cap F^{1/2}=\emptyset,
\end{equation}
so \eqref{tesi lemma unione} follows from
\eqref{dim lemma unione}. Moreover, again by \eqref{dim lemma unione} and \eqref{eq:union}, for every Borel set $C\subseteq E^{1/2}\cap  F^{1/2}$ we have
$$
\int_{C}\nu_E+\nu_F\,d\Haus{\infty-1}=\int_{C\cap (E\cup F)^{1/2}}\nu_{E\cup F}\,d\Haus{\infty-1}=0,
$$
which implies that $\nu_E=-\nu_F$ at $\Haus{\infty-1}$-a.e. point in $E^{1/2}\cap  F^{1/2}$, as desired.
\end{proof}

\section{Appendix: The reduced boundary}

The classical finite-dimensional definition of reduced boundary
\cite{DeG1} is based on the requirements of existence of the limit
\begin{equation}\label{dgredbou}
\nu_E(x):=\lim_{r\downarrow
0}\frac{D\chi_E(B_r(x))}{|D\chi_E|(B_r(x))}
\end{equation}
and modulus of the limit $\nu_E(x)$ equal to 1. It is not hard to
show that points in the reduced boundary are Lebesgue points for the
vector field $\nu_E$, relative to $|D\chi_E|$, hence the proof that
$|D\chi_E|$-almost every point $x$ is in the reduced boundary is based
on Besicovitch covering theorem, a result not available in infinite
dimensions.

In \cite[Definition 7.2]{AMP}, the authors proposed the following
definition of reduced boundary based on the OU semigroup:
\begin{definition}[Gaussian Reduced Boundary]\label{def:reduced} Let $E$ be a Borel set of
finite perimeter in $(X,\gamma)$.
We denote by ${\mathcal F}E$ the set of points $x \in X$ where the limit
\begin{equation}\label{eq:reduced}
\nu_E(x):=\lim_{t\downarrow 0} T_{t}\biggl(\frac{T^*_{t}D_\gamma
\chi_E}{T^*_{t}|D_\gamma \chi_E|}\biggr)(x)
\end{equation}
exists and satisfies $|\nu_E(x)|=1$.
\end{definition}

As observed in \cite[Section 7]{AMP}, a natural open problem is to
prove that $|D_\gamma\chi_E|$ is concentrated on ${\mathcal F}E$.
Here, we show how the soft argument used in the proof of
Proposition~\ref{pweakstar} allows to prove easily the weaker result
\begin{equation}
\label{eq:L1 reduced} \lim_{t\downarrow 0}T_th_t=1\quad \text{in
}L^1(X,|D_\gamma\chi_E|)\qquad\text{with}\qquad
h_t:=\frac{|T^*_{t}D_\gamma \chi_E|}{T^*_{t}|D_\gamma \chi_E|}.
\end{equation}
In particular, we deduce that along any subsequence $(t_i)\downarrow 0$
such that
$$
\sum_i \int_X |T_{t_i}h_{t_i}- 1|\,d|D_\gamma\chi_E|<\infty
$$
it holds
$$
\lim_{i\to\infty}T_{t_i}\biggl(\frac{|T^*_{t_i}D_\gamma
\chi_E|}{T^*_{t_i}|D_\gamma \chi_E|}\biggr)(x)=1\qquad \text{for
$|D_\gamma\chi_E|$-a.e. $x\in X$.}
$$
\begin{proof}[Proof of \eqref{eq:L1 reduced}.]
Set $f_t:=T_t\chi_E$. Arguing as in the proof of Proposition \ref{pweakstar}, the product rule \eqref{product} yields
$$
|D_\gamma(f_t\chi_E)|(X)\leq \int_X f_t \,d|D_\gamma\chi_E| + \int_X h_t \chi_E \,dT_t^*|D_\gamma\chi_E|.
$$
Replacing $E$ by $X\setminus E$ and and $f_t$ by $1-f_t$, we also have
$$
|D_\gamma((1-f_t)\chi_{X\setminus E})|(X)\leq \int_X (1-f_t)
\,d|D_\gamma\chi_E| + \int_X h_t (1-
\chi_E)\,dT_t^*|D_\gamma\chi_E|.
$$
Adding together the two inequalities above, we obtain
\begin{align*}
|D_\gamma(f_t\chi_E)|(X)+|D_\gamma((1-f_t)\chi_{X\setminus E})|(X)
&\leq |D_\gamma\chi_E| (X)+ \int_X h_t\,d T_t^*|D_\gamma\chi_E|\\
&=|D_\gamma\chi_E| (X)+ \int_X T_t h_t\,d|D_\gamma\chi_E|.
\end{align*}
By lower semicontinuity of the total
variation (see Proposition \ref{prop:sci}), letting $t \downarrow 0$ we get
\begin{align*}
2|D_\gamma\chi_E| (X) &\leq \liminf_{t\downarrow 0}\Bigl(|D_\gamma(f_t\chi_E)|(X)+|D_\gamma((1-f_t)\chi_{X\setminus E})|(X)\Bigr)\\
&\leq |D_\gamma\chi_E| (X)+ \liminf_{t\downarrow 0} \int_X T_t h_t\,d|D_\gamma\chi_E|,
\end{align*}
so that
$$
|D_\gamma\chi_E|(X)\leq\liminf_{t\downarrow 0}\int_X T_t
h_t\,d|D_\gamma\chi_E|.
$$
This, combined with the fact that $0\leq T_t h_t\leq 1$ (as $0 \leq
h_t \leq 1$) proves that
$$
\int_X | T_t h_t -1|\,d|D_\gamma\chi_E| =\int_X (1-T_t h_t)\,d|D_\gamma\chi_E| \to 0 \qquad \text{as $t\downarrow 0$},
$$
as desired.
\end{proof}

\end{document}